\documentclass[12pt]{amsart}
\usepackage[utf8]{inputenc}
\usepackage[usenames,dvipsnames]{xcolor}
\usepackage{tikz}
\usepackage{subfigure}
\usepackage{fancyhdr}
\usepackage{txfonts}
\usepackage{epsfig}
\usepackage{mathrsfs}
\usepackage{amssymb}
\usepackage{latexsym}
\usepackage{amsmath}
\usepackage{amsfonts}
\usepackage{amsbsy}
\usepackage{amscd}
\usepackage{amsthm}
\usepackage{verbatim}
\usepackage{color,tikz}
\usetikzlibrary{calc}
\usetikzlibrary{decorations.pathreplacing}
\usetikzlibrary{decorations.markings}
\usetikzlibrary{intersections}

\usepackage{amsbsy,calc}
\usepackage{ifthen,cite}
\usepackage{appendix}
\usepackage{latexsym,amsthm,amssymb,amsmath,amsfonts,graphicx,epstopdf,subfigure,float}
\usepackage{bbm}
\usepackage{pifont}
\usepackage{tikz,xcolor}

\newtheorem{theorem}{Theorem}[section]
\newtheorem{thm}[theorem]{Theorem}
\newtheorem{pro}{Proposition}[section]

\newtheorem{lemma}[pro]{Lemma}
\newtheorem{remark}[pro]{Remark}

\numberwithin{equation}{section}

\def\b{\bigskip}

\def\Q{\mathbb Q}

\def\mathscr{\mathcal }

\def\b0{\mathbf 0}

\def\br{\mathbf r}

\newcommand{\bx}{{\mathbf{x}}}
\newcommand{\by}{{\mathbf{y}}}

\newcommand{\SD}{{\mathcal{D}}}

\newif\ifdraft
\drafttrue
\ifdraft
\else
\fi
\begin{document}

\title{Strict H\"older equivalence of self-similar sets}

\author{Yanfang Zhang}
\address{School of Science, Huzhou University, Huzhou 313000, Zhejiang, China.}
\email{03002@zjhu.edu.cn}

\author{Xinhui Liu$\dag$}
\address{School of Science, Huzhou University, Huzhou 313000, Zhejiang, China.}
\email{xinhuiliu1005@163.com}

\date{\today}
\thanks{This work is supported by Huzhou Natural Science Foundation No. 2022YZ37.}
\thanks{$\dagger$ Corresponding author.}
\maketitle{ }

\begin{abstract}
	The study of Lipschitz equivalence of fractals is a very active topic in recent years. It is natural to ask when two fractal sets are strictly H\"older equivalent. In the present paper, we completely characterize the strict H\"older equivalence for two classes of self-similar sets: the first class is totally-disconnected fractal cubes and the second class is self-similar sets with two branches which satisfy the strong separation condition.
\end{abstract}

\section{Introduction}

Since the poineer works of Falconer and Marsh \cite{FM92} and David and Semes \cite{DS}  around 1990, the study of bi-Lipschitz classification of self-similar sets has become a very active topic and abundant results have been obtained, see\cite{RaoRW12, FanRZ15,LuoL13,RRX06, XiXi10, RuanWX14,RaoZ15, LuoL16, RuanW17, YZ18, MZ20, XiXi20}.

 Two metric spaces  $(X,d_X)$ and $(Y,d_Y)$ are said to be  \emph{ strictly H\"older equivalent}, denoted by $$X\overset{\text{H\"older}}{\sim}Y,$$
		if there is a bijection $f:~X\to Y$, and  constants $s, C>0$ such that
\begin{equation}\label{eq:Holder}
 C^{-1}d_X(x_1,x_2)^{s} \leq d_Y \big( f(x_1),f(x_2)\big ) \leq C d_X(x_1,x_2)^s,~~~~\forall~~x_1,x_2\in X.
 \end{equation}
		In this case, we say $f$ is a \emph{bi-H$\ddot{o}$lder map} with index $s$.
		If $s=1$, we say $X$ and $Y$ are  \emph{Lipschitz equivalent}, denoted by $X\sim Y$, and call $f$ a \emph{bi-Lipschitz map}.

 \begin{remark} \rm We remark that two metric spaces  $(X,d_X)$ and $(Y,d_Y)$ are said to be  \emph{H\"older equivalent},
		if there is a bijection $f:~X\to Y$, and  constants $0<s<1, C>0$ such that
\begin{equation}\label{eq:Holder}
 C^{-1}d_X(x_1,x_2)^{1/s} \leq d_Y \big( f(x_1),f(x_2)\big ) \leq C d_X(x_1,x_2)^s,~~~~\forall~~x_1,x_2\in X.
 \end{equation}
		 Clearly strict H\"older equivalence implies H\"older equivalence. For H\"older equivalence of fractal sets, see
\cite{FM92,  HWYZ23 ,ZX22} \it{etc}.
\rm  For example, Falconer and Marsh \cite{FM92} proved that every self-similar set satisfying
the strong separation condition is H\"older equivalent to the ternary Cantor set.
\end{remark}

A family of contractions $\{\varphi_j\}_{j=1}^{N}$   on $\mathbb{R}^{d}$ is called
an \emph{iterated function system} (IFS).
The unique nonempty compact set $K$ satisfying
$K=\bigcup_{j=1}^N\varphi_j(K)$ is called the attractor of the IFS.
If $\varphi_i(K)\cap \varphi_j(K)=\emptyset$ provided $i\neq j$, then we say the IFS (or $K$)
satisfies the \emph{strong separation condition}.

If the contractions $\{\varphi_j\}_{j=1}^{N}$ are all similitude, then the attractor is called  a \emph{self-similar set} defined by the IFS (cf. \cite{Hutchinson1981}).
In particular, if
$$\left \{\varphi_j\left({x}\right)= \frac{x+d_j}{n}\right\}_{j=1}^N$$
where $n \geq 2, d\geq 1$ are integers and $ \SD=\{d_1,\dots, d_N\}\subset \{0,\dots, n-1\}^d$, then
we denote the  attractor by $K=K(n,\SD)$ and call it a \emph{fractal cube}, see \cite{LLR13, XiXi10}. Figure \ref{fig:cross} is a fractal cube with $d=2,n=5,N=20$.

\begin{figure}[h]
  \centering
  \includegraphics[width=0.36\textwidth]{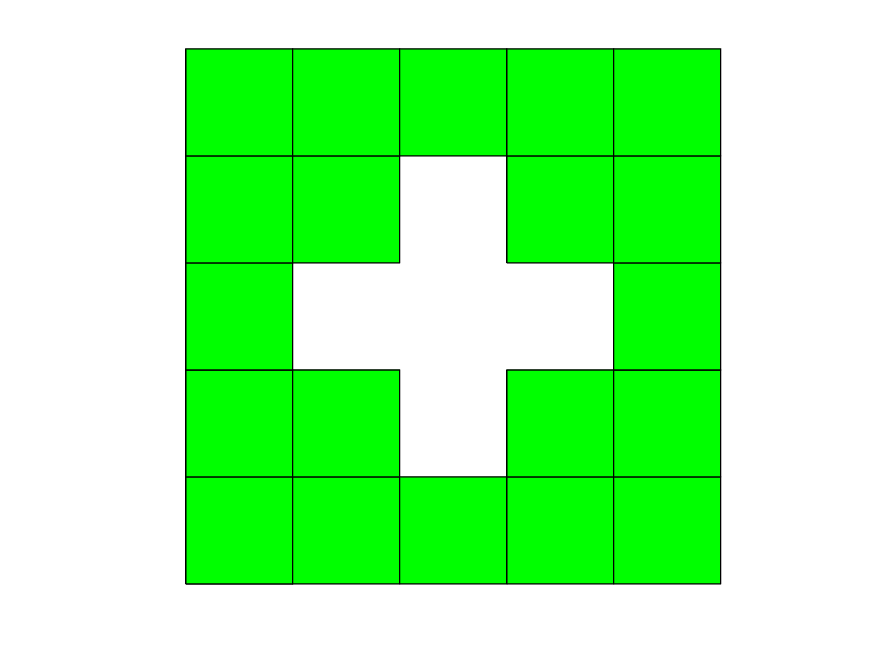}
  \includegraphics[width=0.36\textwidth]{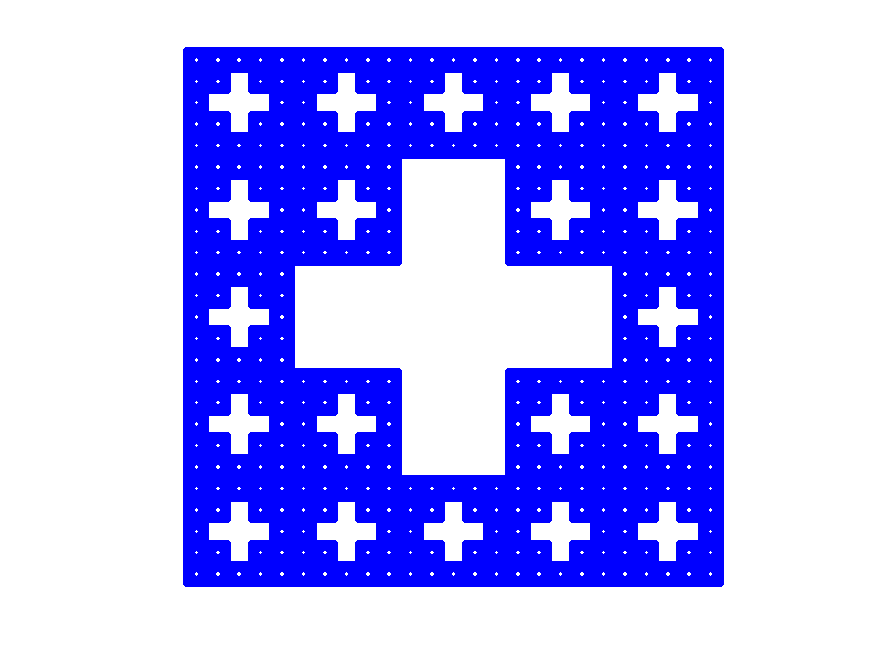}\\
  \caption{A fractal cube: the picture on the left indicates the digit set and the picture on the right is the fractal cube.}
  \label{fig:cross}
\end{figure}

Xi and Xiong \cite{XiXi10} characterized when two totally disconnected fractal cubes are Lipschitz equivalent.

\begin{pro} [\cite{XiXi10}] Let $E=K(n, \SD)$ and $F=K(n', \SD')$ be two totally disconnected fractal cubes.
Then $E$ and $F$ are Lipschitz equivalent if and only if
\begin{equation}
\frac{\log n}{\log n'}=\frac{\log N}{\log N'}\in \Q.
\end{equation}
\end{pro}

 The first result of this paper is following.

\begin{theorem}\label{thm:cube}
	 Let $E=K(n, \SD)$ and $F=K(n', \SD')$ be totally disconnected fractal cubes where
$\SD$ is a digit set with cardinality $N$ and the cardinality of $\SD'$ is $N'$.
Then $E$ and $F$ are strictly H\"older equivalent if and only if
\begin{equation}
 \frac{\log N}{\log N'}\in \Q.
\end{equation}
\end{theorem}

Next, we study strict H\"older equivalence of self-similar sets satisfying the strong separation condition.
Let $r_1,\dots, r_m\in (0,1)$. We use
$${\mathcal S}(r_1,\dots, r_m)$$
 to denote the collection of
self-similar sets with contraction ratios $r_1,\dots, r_m$ and satisfying the strong separation condition.
It is well-known that any two elements in ${\mathcal S}(r_1,\dots, r_m)$ are Lipschitz equivalent, see for instance,
Rao, Ruan and Wang \cite{RaoRW12}.

\begin{theorem}\label{thm:SSC} Let $E\in {\mathcal S}(r_1,\dots, r_m)$ and $F\in {\mathcal S}(t_1,\dots, t_n)$.
Let $s=\dim_H E/\dim_H F$. Let $E'\in {\mathcal S}(r_1^s, \dots, r_m^s)$.
Then $E\overset{\text{H\"older}}{\sim} F$ if and only if $E'\sim F$.
\end{theorem}

Rao, Ruan and Wang \cite{RaoRW12} have characterized when two self-similar sets
with two branches are Lipschitz equivalent.  Based on their result, we completely characterize when two self-similar sets with two branches are strictly H\"older equivalent.

\begin{theorem}\label{thm:two-branch} Let $E\in {\mathcal S}(r_1, r_2)$ and $F\in {\mathcal S}(t_1,t_2)$.
Without loss of generality, assume that $r_1\geq r_2$ and $t_1\geq t_2$.
Then

(i) If $\log r_1/\log r_2\not\in \Q$, then $E\overset{\text{H\"older}}{\sim} F$
if and only if
$$
\frac{\log r_1}{\log t_1}=\frac{\log r_2}{\log t_2}.
$$

(ii) If $\log r_1/\log r_2\in \Q$, then $E\overset{\text{H\"older}}{\sim} F$
if and only if
$$
\frac{\log r_1}{\log r_2}=\frac{2}{3} \text{and} \quad \frac{\log t_1}{\log t_2}=\frac{1}{5},
$$
or the other round $\frac{\log r_1}{\log r_2}=\frac{1}{5},$ $\frac{\log t_1}{\log t_2}=\frac{2}{3}$.
\end{theorem}
 \begin{remark} \rm
The ratios of Theorem \ref{thm:two-branch} arise from the algebraic constraints in the Lipschitz equivalence.
\end{remark}

	\medskip
	
	This article is organized as follows:
	In Section 2, we  study the strict H\"older equivalence of fractal cubes and prove Theorem \ref{thm:cube} there.
	In Section  3,  we investigate the strict H\"older equivalence of  self-similar sets satisfying the strong separation condition and prove Theorem \ref{thm:SSC} and Theorem \ref{thm:two-branch} there.

\section{\textbf{Strict H\"older equivalence of  fractal cubes} }\label{symbol}

First, we  recall some basic definitions and facts about symbolic spaces.

Let $N\geq 2$ be an integer. Set
$$
\Omega_N=\{0,1,\dots, N-1\}^\infty.
$$
Let $\bx=(x_j)_{j=1}^\infty, \by=(y_j)_{j=1}^\infty\in \Omega_N$. Let $\bx\wedge\by$ be the maximal common prefix of $\bx$ and $\by$. For a word $x_1\dots x_n$, we denote $|x_1\dots x_n|=n$ to be its length.

Let $r\in (0,1)$. We define a metric $\rho_r$ on $\Omega_N$ as
$$
\rho_r(\bx,\by)=r^{|\bx\wedge \by|}.
$$
Then $(\Omega_N, \rho_r)$ is a metric space. It is well known that the Hausdorff dimension of this space is
$$
\dim_H(\Omega_N, \rho_r)=\frac{\log N}{-\log r}.
$$

\begin{thm}[\cite{RaoRW12}]\label{thm:RRW}
 Two symbolic spaces $(\Omega_N, \rho_r)$ and $(\Omega_{N'}, \rho_{r'})$ are Lipschitz equivalent if and only if
 $$
 \frac{\log r}{\log r'}=\frac{\log N}{\log N'}\in \Q.
 $$
\end{thm}

 We show that

\begin{thm}\label{thm:NN'}
  Two symbolic spaces $(\Omega_N, \rho_r)$ and $(\Omega_{N'}, \rho_{r'})$ are strictly H\"older equivalent if and only if
 $$
 \frac{\log N}{\log N'}\in \Q.
 $$
\end{thm}

\begin{proof} First, we claim that  $(\Omega_N, \rho_r)$ and $(\Omega_{N}, \rho_{r'})$ are strictly H\"older equivalent.
Let $f(x)=x$ be the identity map from $(\Omega_N, \rho_r)$ to $(\Omega_{N}, \rho_{r'})$.
Let $\bx,\by\in \Omega_N$, then
$$
\rho_r(\bx,\by)=r^{|\bx\wedge\by|}, \quad \rho_{r'}(\bx,\by)=(r')^{|\bx\wedge\by|}.
$$
Therefore,
$$
\rho_{r'}(f(\bx),f(\by))=\rho_r(\bx,\by)^s
$$
where $s=\log r'/\log r$. Our claim is proved.

Next, we will complete the proof of the theorem in two steps.

First, suppose $\frac{\log N}{\log N'}=p/q\in \Q$, where $p,q$ are coprime integers, then $N^q=(N')^p$. Therefore,
$$
(\Omega_N, \rho_r)\sim (\Omega_{N^q}, \rho_{r^q})\overset{\text{H\"older}}{\sim}
(\Omega_{N^q}, \rho_{(r')^p})= (\Omega_{(N')^p}, \rho_{(r')^p})\sim (\Omega_{N'}, \rho_{r'}),
$$
that is, $(\Omega_N, \rho_r)$ and $(\Omega_{N'}, \rho_{r'})$ are strictly H\"older equivalent.

Secondly, suppose $(\Omega_N, \rho_r)$ and $(\Omega_{N'}, \rho_{r'})$ are strictly H\"older equivalent. Then
there exists $s, C>0$ such that
$$
C^{-1}  r^{s|\bx\wedge \by|}\leq   {r'}^{|f(\bx)\wedge f(\by)|}\leq C r^{s|\bx\wedge \by|}.
$$
It follows that
$(\Omega_N, \rho_{r^s})$ and $(\Omega_{N'}, \rho_{r'})$ are Lipschitz equivalent through the identity map. So by Theorem \ref{thm:RRW}, we have
$$
\frac{\log N}{\log N'}\in \Q.
$$
This completes the proof.
\end{proof}


The following lemma is  proved by Xi and Xiong \cite{XiXi10}.

\begin{lemma}[\cite{XiXi10}] \label{lem:XiXi} Let $E=K(n, \SD)$ be a totally disconnected fractal cube. Denote $N=\#\SD$. Then
$E\sim (\Omega_N, \rho_{1/n})$.
\end{lemma}

Now we are in the position to prove Theorem \ref{thm:cube}.

\begin{proof}[Proof of Theorem \ref{thm:cube}.] Let $E=K(n, \SD)$ and $F=K(n', \SD')$. By Lemma \ref{lem:XiXi},
$E\sim (\Omega_N, \rho_{1/n})$ and $F\sim (\Omega_{N'}, \rho_{1/n'})$. Hence, the theorem follows from
Theorem  \ref{thm:NN'}.
\end{proof}

\section{\textbf{Proof of Theorem \ref{thm:SSC} and Theorem \ref{thm:two-branch}}}
 In this section, we investigate strict H\"older equivalence between
two self-similar sets satisfying the strong separation condition.

Given $N\geq 2$ and $\br=(r_1,\dots, r_N)\in (0,1)^N$, we define
$$\br_{a_1\dots a_k}=r_{a_1}\cdots r_{a_k}$$
for $a_1,\dots, a_k\in \{1,\dots, N\}$.
We define a metric $\rho_{\br}$ on $\Omega_N$ as follows:
\begin{equation}
\rho_{\br}(\bx,\by)=\br_{\bx\wedge \by}, \quad \bx, \by\in \Omega_N.
\end{equation}
The following result is folklore.

\begin{lemma}[\cite{FM92}] \label{lem:one} Let $E\in {\mathcal S}(r_1,\dots, r_N)$, then
$E\sim (\Omega_N, \rho_{\br})$.
\end{lemma}

The following lemma is obvious, and we omit its proof.

\begin{lemma}\label{lem:two} Let $s>0$. Then
$$
(\Omega_N, \rho_{(r_1,\dots, r_N)})  \overset{\text{H\"older}}{\sim}(\Omega_N, \rho_{(r_1^s,\dots, r_N^s)}).
$$
\end{lemma}

\begin{proof}[\textbf{Proof of Theorem \ref{thm:SSC}.}]  Let $E\in {\mathcal S}(r_1,\dots, r_m)$ and $F\in {\mathcal S}(t_1,\dots, t_n)$.
Let $s=\dim_H E/\dim_H F$. Let $E'\in {\mathcal S}(r_1^s, \dots, r_m^s)$.

Suppose $E\overset{\text{H\"older}}{\sim} F$. Then by Lemma \ref{lem:one} and Lemma \ref{lem:two}, we have
\begin{equation}\label{eq:bridge}
E'\overset{\text{H\"older}}{\sim} (\Omega_m, \rho_{(r_1^s,\dots, r_m^s)})\overset{\text{H\"older}}{\sim} (\Omega_m, \rho_{(r_1,\dots, r_m)}) \overset{\text{H\"older}}{\sim} E.
\end{equation}
It follows that $E'\overset{\text{H\"older}}{\sim} F$. Finally, since $\dim_H E'=\dim_H F$, we obtain that $E'\sim F$.

Suppose $E'\sim F$. Again by \eqref{eq:bridge}, we have $E\overset{\text{H\"older}}{\sim} F$.
The theorem is proved.
\end{proof}

Rao, Ruan and Wang \cite{RaoRW12} completely characterized when two self-similar sets with two branches and satisfying
the strong separation condition are Lipschitz equivalent.

\begin{theorem}[\cite{RaoRW12}] \label{thm:Lip} Let $E\in {\mathcal S}(r_1, r_2)$ and $F\in {\mathcal S}(t_1,t_2)$.
Assume $r_1\geq r_2$ and $t_1\geq t_2$ without loss of generality.
Then

(i) If $\log r_1/\log r_2\not\in \Q$, then $E {\sim} F$
if and only if
$$
 r_1 = t_1 \text{ \ and \ } r_2=t_2.
$$

(ii) If $\log r_1/\log r_2\in \Q$, then $E {\sim} F$
if and only if there is a real number $0<\lambda<1$ such that
$$
r_1=\lambda^2, r_2=\lambda^3, t_1=\lambda, t_2=\lambda^5,
$$
or the other round $t_1=\lambda^2, t_2=\lambda^3, r_1=\lambda, r_2=\lambda^5$.
\end{theorem}

\begin{proof}[\textbf{Proof of Theorem \ref{thm:two-branch}.}]Let $E\in {\mathcal S}(r_1, r_2)$ and $F\in {\mathcal S}(t_1,t_2)$.
Assume $r_1\geq r_2$ and $t_1\geq t_2$ without loss of generality. Let $s=\dim_H E/\dim_H F$.

\textit{Case 1.}  $\log r_1/\log r_2\not\in \Q$.

Suppose   $E\overset{\text{H\"older}}{\sim} F$. Let $E'\in {\mathcal S}(r_1^s, r_2^s)$, then
$E'\sim F$ by Theorem \ref{thm:SSC}. So, by Theorem \ref{thm:Lip}, we have
$$
r_1^s=t_1 \text{ and } r_2^s=t_2,
$$
which imply that $\frac{\log r_1}{\log t_1}=\frac{\log r_2}{\log t_2}=s$.

On the other hand, suppose $\frac{\log r_1}{\log t_1}=\frac{\log r_2}{\log t_2}=\delta$.
Then $r_1=t_1^{\delta}$ and $r_2=t_2^\delta$.
Let $F'\in {\mathcal S}(t_1^\delta,t_2^\delta)$, then
$$E\sim F'\overset{\text{H\"older}}{\sim} F,$$
where the last relation is due to Lemma \ref{lem:one} and Lemma \ref{lem:two}. The theorem is proved in this case.

\textit{Case 2.}  $\log r_1/\log r_2\in \Q$.

Suppose $E\overset{\text{H\"older}}{\sim} F$. Let $E'\in {\mathcal S}(r_1^s, r_2^s)$, then
$E'\sim F$ by Theorem \ref{thm:SSC}. By Theorem \ref{thm:Lip}, we have
\begin{equation}\label{eq:lambda}
r_1^s=\lambda^2, \quad r_2^s=\lambda^3, \quad t_1=\lambda, \quad   t_2=\lambda^5,
\end{equation}
or the other round, it follows that
\begin{equation}\label{eq:2/3}
\frac{\log r_1}{\log r_2}=\frac{2}{3}, \quad \frac{\log t_1}{\log t_2}=\frac{1}{5},
\end{equation}
or the other round $\frac{\log r_1}{\log r_2}=\frac{1}{5},$ $\frac{\log t_1}{\log t_2}=\frac{2}{3}$.

On the other hand, suppose  \eqref{eq:2/3} holds.
Then  \eqref{eq:lambda} holds for some $s>0$.
Let $E'\in {\mathcal S}(r_1^s,r_2^s)$, then
$$F\sim E'\overset{\text{H\"older}}{\sim} E.$$
The theorem is proved.
\end{proof}

\end{document}